\documentclass[12pt]{amsart}
\usepackage{amsmath}
\usepackage{amssymb}
\usepackage{epsfig}
\usepackage{graphicx}
\usepackage{hyperref}
\numberwithin{equation}{section}

\usepackage[OT2,T1]{fontenc}
\DeclareSymbolFont{cyrletters}{OT2}{wncyr}{m}{n}
\DeclareMathSymbol{\Sha}{\mathalpha}{cyrletters}{"58}

\input xy
\xyoption{all}

\calclayout
\allowdisplaybreaks[3]

\theoremstyle{plain}
\newtheorem{prop}{Proposition}

\newtheorem{theo}[prop]{Theorem}

\newtheorem{lemm}[prop]{Lemma}

\theoremstyle{definition}

\newtheorem{rema}[prop]{Remark}

\newtheorem{exam}[prop]{Example}

\makeatother
\makeatletter

\newcommand{\A}{{\mathbb A}}

\newcommand{\Z}{{\mathbb Z}}

\newcommand{\G}{{\mathbb G}}
\newcommand{\PP}{{\mathbb P}}

\newcommand{\bA}{\mathbb A}
\newcommand{\bC}{\mathbb C}
\newcommand{\bF}{\mathbb F}

\newcommand{\bP}{\mathbb P}

\newcommand{\bZ}{\mathbb Z}

\newcommand{\cD}{\mathcal D}

\newcommand{\cM}{\mathcal M}
\newcommand{\cO}{{\mathcal O}}

\newcommand{\cS}{\mathcal S}

\newcommand{\cV}{\mathcal V}
\newcommand{\cX}{\mathcal X}

\newcommand{\fS}{\mathfrak S}

\DeclareMathOperator{\Spec}{Spec}

\DeclareMathOperator{\Pic}{Pic}
\DeclareMathOperator{\Br}{Br}

\newcommand{\IIJ}{\operatorname{IJ}}

\newcommand{\NS}{\operatorname{NS}}

\newcommand{\ra}{\rightarrow}

\author{Brendan Hassett}
\address{Department of Mathematics\\
Brown University \\
151 Thayer Street \\
Providence, Rhode Island 02912 \\
USA}
\email{bhassett@math.brown.edu}

\author{Andrew Kresch}
\address{Institut f\"ur Mathematik \\
Universit\"at Z\"urich \\
Winterthurerstrasse 190 \\
CH-8057 Z\"urich \\
Switzerland}
\email{andrew.kresch@math.uzh.ch}

\author{Yuri Tschinkel}
\address{Courant Institute\\
                New York University \\
                New York, NY 10012 \\
                USA }
\email{tschinkel@cims.nyu.edu}

\address{Simons Foundation\\
160 Fifth Avenue\\
New York, NY 10010\\
USA}

\title{Stable rationality in smooth families of threefolds}

\begin{document}

\begin{abstract}
We exhibit families of smooth projective threefolds with both
stably rational and non stably rational fibers.
\end{abstract}

\date{February 16, 2018}

\maketitle

\section{Introduction}
Rationality, and thus as well stable rationality, is a deformation invariant property of
smooth complex projective curves and surfaces.
We now know that rationality and stable rationality specialize in
families of 
smooth complex
projective varieties of arbitrary dimension \cite{NS}, \cite{KT}.
In dimension at least four, there exist families of smooth complex projective
varieties with both rational and 
non stably rational fibers \cite{HPT}, \cite{HPT-quadric}, \cite{schreieder-1}, \cite{schreieder-2}.
The case of relative dimension three is open.  

In this note, we exhibit a family of smooth complex threefolds with
both stably rational and non stably rational fibers.

\begin{theo}
\label{theo:headline}
There exists a smooth projective family $\psi:\cV \ra B$ of complex 
threefolds over a connected curve $B$, such that for some 
$b_0\in B$ the fiber $\cV_{b_0}:=\psi^{-1}(b_0)$ 
is stably rational and the very general fiber $\cV_b:=\psi^{-1}(b)$ 
is not stably rational. In particular, 
stable rationality is not a deformation invariant of smooth
complex projective threefolds. 
\end{theo}

Our examples originate from the first examples of
non-rational but stably rational varieties \cite{BCTSSD}.
The key ingredient is a class of smooth projective surfaces over
non-closed fields $k$ that are stably rational but not rational
(see Section~\ref{sect:recall}).
For $k=\mathbb C(t)$, we obtain fibered threefolds 
$Y \ra \bP^1$ that are stably rational over $\bC$.
Using the technique of intermediate Jacobians, one can
show that some of these fibered threefolds are irrational.

We work in a similar vein, considering threefolds fibered in
stably rational surfaces. Instead of the intermediate
Jacobian, we employ the 
groundbreaking work on specialization of 
stable rationality by Voisin \cite{voisin}, and its subsequent developments in 
\cite{CT-Pirutka}, \cite{totaro}, \cite{HKT}, \cite{HT-Fano}.

Here is a more detailed summary of the contents of this
paper: 
We review the key class of stably rational non-rational surfaces in
Section~\ref{sect:recall} and recast their Galois-theoretic properties,
when defined over $\bC(t)$,
in terms of finite covers of nodal curves in 
Section~\ref{sect:geom}. Section~\ref{sect:construct} sketches the
construction of the families of threefolds.
The analysis of Section~\ref{sect:DCB} may be
of independent interest: How can we construct families of standard
conic bundles over a prescribed family of ramification data? 
Section~\ref{sect:proof} establishes the failure of stable rationality
for general deformations of the examples constructed previously.
Finally, we explain why these tools fail to yield stably rational
cubic threefolds in Section~\ref{sect:limit}.

\

\noindent
{\bf Acknowledgments:}  The first author was partially supported by NSF grants
1551514 and 1701659 and the
third author by NSF grant 1601912. We are grateful to Dan Abramovich and
Jean-Louis Colliot-Th\'el\`ene for discussions about this work.

\section{Recollections on stably rational non-rational surfaces}
\label{sect:recall}

Let $k$ be a field of characteristic zero with absolute Galois group $G_k$.
For us, a {\em Ch\^atelet surface} is 
\begin{equation} \label{eqn:1}
V=\{ (x,y,z) \, \mid \,  y^2-az^2 = f(x) \} \subset \bA^3,
\end{equation}
where $f\in k[x]$ is a cubic polynomial with Galois group
the symmetric group  $\fS_3$ and $a=\mathrm{disc}(f)$.
In particular, $a$ is not a square in $k$.
These exist whenever $k$ admits extensions with Galois
group $\fS_3$.

Let $F(x,w) \in k[x,w]$ be a homogeneous quartic form with
$F(x,1)=f(x)$; note that $w\mid F(x,w)$. The compactification
$$\hat{V} = \{(w:x:y:z) \, \mid \,  y^2-az^2 = F(x,w) \} \subset
\bP(1,1,2,2)$$
has two ordinary singularities $(0:0:\pm \sqrt{a}:1)$.
This admits a natural embedding as a complete intersection
of two quadrics in $\bP^4$. Writing
$$u_0=x^2,\quad u_1=xw,\quad u_2=w^2,$$
we may express
$$\bP(1,1,2,2)= \{u_1^2=u_0u_2 \} \subset \bP^4_{u_0,u_1,u_2,y,z},$$
so that $\hat{V} \subset \bP(1,1,2,2)$ is cut out by a quadratic form.

\begin{theo}\cite{BCTSSD}, \cite{ShB}
\label{theo:0}
The Ch\^atelet surface $V$ is stably rational but not rational over $k$. 
\end{theo}

Remarkably, essentially all stably rational surfaces with these
invariants arise in this way:

\begin{theo} \cite[Thm.~E]{KST} \label{theo:KST}
Let $W$ be a quartic del Pezzo surface over $k$ such that
$\Pic(W) \simeq \bZ \oplus \bZ$ and $W$ is stably rational but 
not rational. Then $W$ is birational over $k$ to a 
Ch\^atelet surface.
\end{theo}

The argument involves classifying possible Galois actions,
which are constrained by the fact that any
stably rational surface $W$ satisfies the following
condition: the 
N\'eron-Severi group $\NS(\bar{W})$
is a direct summand of a permutation Galois module. This
implies that for
each closed subgroup $H\subset G_k$, the Galois cohomology 
$$H^1(H,\NS(\bar{W}))=0.$$
Such actions have been classified for degree $4$
del Pezzo surfaces in \cite[p.~15]{KST} and for degrees $3$, $2$, and $1$
in \cite{TY}.
When the surface admits a conic bundle structure 
$W\ra \bP^1$ over $k$ then the Galois action can be
described as follows:

Generally, a conic bundle with $n$ degenerate fibers has Galois
group contained in the 
Weyl group $\mathsf W(D_n)$, realized as permutations of the irreducible
components of the degenerate fibers. We express this as the subgroup of signed 
permutations having an even number of minus signs,
i.e., as a subgroup of 
$$
(\bZ/2\bZ)^n \rtimes \fS_n.
$$
Here $c_i$, for $i=1$, $\dots$, $n$, denotes the identity permutation
with minus sign in the $i$th position,
i.e., exchanging the components of the $i$th fiber.

The case of interest to us is $n=4$: 
For Ch\^atelet surfaces as in \eqref{eqn:1},
the relevant fibration is $\tilde{V} \ra \bP^1_{x,w}$.
The Galois action on exceptional curves in the fibers
corresponds to a representation
$$
\rho:G_k \ra \mathsf W(D_4)
$$
with image 
$$\left<(23)c_1c_2c_3c_4,(34)c_1c_2c_3c_4\right>,$$
which is isomorphic to $\fS_3$. This reflects the fact that the
discriminant quadratic extension of the cubic $f$ splits the
components of each degenerate fiber. 

We return to our discussion of Theorem~\ref{theo:KST}:
{\em Any} standard minimal
conic bundle $W\ra \bP^1$ with four degenerate fibers,
satisfying the Galois-theoretic conditions above has
Galois representation $\rho$.
Such $W$ map anticanonically either to quartic del Pezzo surfaces
or to Ch\^atelet surfaces in $\PP^4$.
But there always exists a birational map over $k$
$$W \stackrel{\sim}{\dashrightarrow} \tilde{V}.$$
Our approach exploits the flexibility of passing between these models.

\section{Geometric analysis of Galois representation}
\label{sect:geom}
We now assume $k=\bC(t)$.
In contrast to the approach of \cite{BCTSSD}, we work with the
surfaces $W\ra \bP^1$ with Galois action $\rho$
as in Section~\ref{sect:recall}, rather than just the 
Ch\^atelet model.  
Here we interpret the properties of $\rho$ in geometric terms.

We consider models for $W$, regarding $\bC(t)$ as the function field
of $\bP^1$. These are conic bundles over ruled surfaces
$$\pi:X \ra \bF \stackrel{\varphi}{\ra} \bP^1$$
satisfying the following:
\begin{itemize}
\item{$X$ and $\bF$ are smooth and projective;}
\item{$\bF\ra \bP^1$ is generically
a $\bP^1$-bundle and
$X\ra \bF$ is a conic bundle;}
\item{the degeneracy locus of the conic bundle is a nodal curve $D\subset \bF$
with smooth irreducible components $C$ and $L$, with $L$ a section
and $C$ a simply-branched trisection of $\varphi$;}
\item{the associated double cover $\tilde{D}\ra D$,
parametrizing irreducible
components of the degenerate fibers of $W\ra \bP^1$, ramifies over $C\cap L$;}
\item{the Galois action of $\tilde{D} \ra \bP^1$ coincides with
$\rho$.}
\end{itemize}

We recall, by \cite[Thm.\ 1]{artinmumford} and \cite[Thm.\ 5.7]{sarkisov2},
over a smooth projective rational surface such as $\bF$,
the data of a nodal discriminant curve and,
for each component, a nontrivial degree $2$ cover ramified precisely
over the nodes,
determine a standard conic bundle,
uniquely up to birational isomorphism.

We express the Galois conditions in geometric terms:
Let $g$ be the genus of $C$, $r_1,\ldots,r_{2g+4} \in C$
the ramification points of $C\ra \bP^1$,
$p_1,\ldots,p_{2g+4}\in C$ the residual points in each fiber,
$\tilde{C} \ra C$ the discriminant cover, and 
$p'_1,\ldots,p'_{2g+4} \in \tilde{C}$ its ramification points.
Let $\tilde{L}\ra L$ be the double cover,
$w_1,\ldots,w_{2g+4}$ its branch points, and
$w'_1,\ldots,w'_{2g+4}\in \tilde{L}$ the ramification points.
The cover $\tilde{D}\ra D$ is the admissible double cover
$$
\tilde{L} \cup_{w'_i=p'_i} \tilde{C} \ra L \cup_{w_i=p_i} C.
$$
In particular,
$$
C\cap L = \{p_1=w_1,\ldots,p_{2g+4}=w_{2g+4}\}=:\{q_1,\ldots,q_{2g+4}\}.$$

Conversely, conic fibrations $X \ra \bF$ with degeneracy 
(i.e., ramification) data as above necessarily have 
Galois representation $\rho$, as the action on the
N\'eron-Severi group can be read off from the induced
permutation of the irreducible components of the four degenerate
conic fibers.

Given such a $\pi:X \ra \bF\ra \bP^1$, the fiber $W$ over the generic
point of $\bP^1$ is as in Section~\ref{sect:recall}, thus
is stably rational over $\bC(t)$ 
by Theorem~\ref{theo:KST}. It follows that $X$ is stably rational
over $\bC$.

\begin{rema}
The key constraint is that the images of the branch loci of
$\tilde{C} \ra C$ and $\tilde{L} \ra L$ in $\bP^1$ coincide
$$\varphi(\{p_1,\ldots,p_{2g+4}\})=\varphi(\{w_1,\ldots,w_{2g+4}\}).$$
Once $C\ra \bP^1$ is given, there is a {\em canonical} choice
of $\tilde{L} \ra L$, determined by the discriminant quadratic extension.
\end{rema}

\begin{rema}
The double cover $\tilde{D} \ra D$
is admissible and its Prym variety $P$ is a principally
polarized abelian variety of dimension $3g+2$. If $g>0$ then $P$
is not a product of Jacobians \cite[Rem.~7]{BCTSSD}. The intermediate
Jacobian $\IIJ(X)=P$.  
Hence $X$ is not rational.
\end{rema}

\section{Construction of examples}
\label{sect:construct}
\subsection{Embedding of the degeneracy curve}
\label{subsect:embed}
We use the machinery of Section~\ref{sect:geom}
to construct examples of stably rational
threefolds. 
The simplest possible case of interest is $g=1$.

We start by fixing $f:C \ra \bP^1$, a simply branched triple cover,
with ramification points $r_1,\ldots,r_6$ and residual points $p_1,\ldots,p_6$.
In other words, $p_i$ is the residual to $r_i$ in $f^{-1}(f(r_i))$.
Set $L=\bP^1$ and 
glue $p_i$ to $f(r_i)$ to obtain $D$. 
We use $q_i,i=1,\ldots,6,$
for the nodes of $D$ arising from gluing $p_i$ and $f(r_i)$. 
Let $\varphi:D\ra \bP^1$ 
denote the resulting degree $4$ cover.

Our goal is to embed $D$ in the simplest possible rational surface.
We may interpret $D$ as a stable curve of genus six. 

\begin{rema}
We show in Section~\ref{sect:limit}
that $D$ cannot be embedded as a quintic plane curve.
Thus our approach does not yield stably rational cubic threefolds.
Recall that projecting from a line in a cubic threefold gives a conic bundle
over $\bP^2$ with quintic degeneracy locus, and almost all quintic
plane curves arise in this way.
\end{rema}

The generic stable curve $D'$ of genus six
arises as a bi-anticanonical curve in a quintic del Pezzo surface 
$S$. 
Realizing $S$ as the blowup of $\bP^2$ in four points,
$D'$ may be interpreted as a sextic plane curve with four nodes.
This motivates the following technical result:

\begin{prop} \label{prop:embed}
Let $D$ be as above. Then there exist:
\begin{itemize}
\item{an embedding $C\hookrightarrow \bP^2$ as a cubic plane curve;}
\item{a morphism $\iota:L\ra \bP^2$ birational onto a cubic curve 
singular at $s_4$;}
\end{itemize}
satisfying the following:
\begin{itemize}
\item{projection from $s_4$ 
$$\bP^2 \dashrightarrow \bP^1$$
restricts to $\varphi$ on $C$;}
\item{the intersection
$$C \cap \iota(L) \supset \{q_1,\ldots,q_6 \}$$
and
the residual points of intersection $s_1,s_2,s_3$ are
collinear.}
\end{itemize}
\end{prop}
\begin{proof}
The embedding of $C$ arises from the linear series associated with
$\varphi^*\cO_{\bP^1}(1)|C$. Let $s_4 \in \bP^2$ be the point inducing
the triple cover $C\ra \bP^1$.

We recall some classical terminology \cite{DolgClassical}:
The ramification of $C\ra \bP^1$ is given by the intersection
of $C$ with its {\em polar conic} with respect to $s_4$. The six
residual points are given by the intersection of $C$ with its
{\em satellite conic} with respect to $s_4$ \cite[p.~157]{DolgClassical}.

Choose a cubic plane curve $C_{\mathrm{sing}}$ with double point at $s_4$ and 
containing the intersection 
of $C$ with its
satellite conic. We claim this is irreducible. Otherwise, 
one of its irreducible components would have to be the satellite conic,
which is precluded by:
\begin{lemm} \label{lemm:satellite}
Fix a smooth plane cubic $C\subset \bP^2$ and a point
$p \in \bP^2 \setminus C$. Then the satellite conic 
for $C$ with respect to $p$ does not contain $p$.
\end{lemm}
\begin{proof} 
Let $P_p(C)$ be the polar conic, $S_p(C)$ the satellite conic,
and $P_p(P_p(C))$ 
the polar line of the polar conic, which joins the two ramification points of the double cover
$$        P_p(C) \ra \bP^1$$
induced by projection from $p$.

By \cite[p.~157, Exer.~3.19]{DolgClassical}, we have
$$Z :=  S_p(C) \cap P_p(C)$$
consists of two points, each of
multiplicity two, which coincide with 
$$P_p(P_p(C))\cap P_p(C).$$
There is a unique plane conic
containing $Z$ and $p$ ---
the union of the two tangents to $S_p(C)$ through $p$. 
The same conclusion can be obtained through a direct computation of
the equations for $S_p(C)$ and $P_p(C)$ in \cite[Example~2.2]{MMW}.
This proves the Lemma.
\end{proof}

We return to the argument for Proposition~\ref{prop:embed}.
Fix $\iota:L \ra C_{\mathrm{sing}} \subset \bP^2$ to be 
normalization.
The base
locus of the pencil of cubics $\left<C,C_{\mathrm{sing}}\right>$
has six points failing to impose
independent conditions on quadrics. The residual three points
$\{s_1, s_2, s_3\}$
fail to impose independent conditions on linear forms, hence are
collinear. 
\end{proof}

For our application, we want $C_{\mathrm{sing}}$ to be a nodal cubic.
A concrete example over a finite field shows that this is possible.

\begin{exam}
We work over the finite field $\mathbf{F}_p$, with $p=13$.
Let $C$ be the cubic curve
$$
x^3+5x^2y+12x^2z+7xy^2+7xyz+3xz^2+10y^3+y^2z+2yz^2+6z^3=0.
$$
Put
$$
s_4:=(1:0:0)\in \bP^2(\mathbf F_p)
$$
and consider the satellite conic associated to $C$ and $s_4$.
The intersection points with $C$ (residual points for the projection from $s_4$) are:
\begin{multline}
\label{eqn:res}
\left\{(4:10:1), (7:12:1), (6:1:1), (2:7:1), (2:4:1), (4:1:0)\right\}.
\end{multline}
The cubic curve $L$, given by
$$
xy^2+2xyz+11xz^2+9y^3+y^2z+10yz^2=0,
$$
is nodal at $s_4$ and passes through the residual points in \eqref{eqn:res}.
\end{exam}

Let $S_0$ denote the blowup of $\bP^2$ at four points $\{s_1,s_2,s_3,s_4\}$,
where $s_1$, $s_2$, $s_3$ are collinear and $s_4$ is generic. 
Projection from $s_4$ induces a morphism
$$\Phi:S_0 \ra \bP^1.$$
Proposition~\ref{prop:embed} gives an embedding $D \hookrightarrow S_0$ as a
bi-anticanonical curve. 

\subsection{The family of surfaces}
\label{subsect:surfaces}

Let $\cS \ra B$ be a smooth projective family of surfaces,
where $S_0=\cS_{b_0}$ is the degenerate quintic del Pezzo surface as
in Section \ref{subsect:embed} and
$\cS_b,b\neq b_0$ is a smooth quintic del Pezzo surface $S$.
The deformation space of $S_0$ is smooth, and the dimension
of the bi-anticanonical linear series
remains constant in the family.
So, there exists a family of nodal curves
$\cD \ra B$ embedded in $\cS$ over $B$,
where $\cD_0$ is as constructed in Section~\ref{subsect:embed},
and $\cD_b$ is smooth for $b\ne b_0$.

\subsection{The conic bundles in the family}
\label{subsect:family}
The analysis of Section~\ref{sect:geom} implies that a conic bundle $X_0 \ra S_0$
with ramification data given by $\tilde{D} \ra D$ is stably rational. Indeed,
$\Phi|D=\varphi$, by construction, thus the generic fiber of the composition
$$X_0 \ra S_0 \ra \bP^1$$
is a stably rational degree $4$ del Pezzo surface over $\bC(\bP^1)$.
Such an $X_0$ is necessarily stably rational over $\bC$.

Let $X\ra S$ denote a standard conic bundle degenerate 
over a generic bi-anticanonical
divisor in $S$.  
The challenge is to fit the conic bundles
$X$ and $X_0$ into a family
$$
       \varpi: \cX \ra \cS \ra B.
$$
Indeed, we make choices in constructing a standard conic bundle from
its ramification data --- the resulting threefold is unique only up to
birational equivalence, although the birational
maps between the various models are well-understood \cite{Sarkisov,sarkisov2}.
It remains to show that these choices can be made coherently in
one-parameter families; Theorem~\ref{thm.defo} shows this is possible.

\section{Deformation of conic bundles}
\label{sect:DCB}

Let $k$ be an algebraically closed field of characteristic different from $2$
and $(B,b_0)$, a pointed curve.
We consider a smooth projective family of rational surfaces $\mathcal{S}\to B$,
irreducible divisor $\mathcal{D}\subset \mathcal{S}$, smooth except for
ordinary double points at $q_1$, $\dots$, $q_r$ in the fiber
$\mathcal{S}_{b_0}$, and a degree $2$ covering
$\widetilde{\mathcal{D}}\to \mathcal{D}$, unramified away $q_1$, $\dots$, $q_r$
with $\widetilde{\mathcal{D}}$ smooth.
We assume that $\mathcal{D}_{b_0}$ has a node at $q_1$, $\dots$, $q_r$
and is otherwise smooth.

The goal is to show (Theorem \ref{thm.defo} in Section \ref{sec.defo})
that after replacing $(B,b_0)$ by an \'etale neighborhood
there is a smooth family of standard conic bundles $\mathcal{V}\to B$
such that the birational type of the fiber over every point $b\in B$
corresponds to the ramification data
$\widetilde{\mathcal{D}}_b\to \mathcal{D}_b$.

\subsection{An algebraic group and some of its representations}
\label{sec.representations}
Working over $\Spec(\Z)$, we consider the classifying stack
$B\Z/2\Z$ with \'etale cover
\[ \pi\colon \Spec(\Z)\to B\Z/2\Z, \]
sending a scheme $T$ to the trivial torsor $T\times \Z/2\Z\to T$.

The stack $B\G_m$ classifies $\G_m$-torsors, or equivalently, line bundles.

A general construction is the \emph{restriction of scalars} along a
proper flat morphism of finite presentation, applicable to
algebraic stacks, locally of finite presentation
with affine diagonal \cite{hallrydh}.
We need
$\pi_*B\G_m$, the restriction of scalars of $B\G_m$
under the morphism $\pi$:
\[ \pi_*B\G_m\cong BH,\qquad
H:=\left\{\begin{pmatrix}{*} & 0 \\ 0 & {*} \end{pmatrix}\right\}
\cup \left\{\begin{pmatrix} 0 & {*} \\ {*} & 0 \end{pmatrix}\right\}
\subset GL_2. \]
Indeed, an $H$-torsor $E\to T$ determines a $\Z/2\Z$-torsor
$S\to T$ (connected components of fibers), such that
$T\times_SE$ admits a canonical reduction of structure group
to $\G_m\times \G_m$, i.e., is determined by a pair of line bundles
$(L,L')$ on $S$; to $E\to T$ we associate $S\to T$ and $L$.
Conversely, given $S\to T$ and $L$ and letting $\sigma\colon S\to S$ denote
the covering involution, we associate to $(L,\sigma^*L)$ a
$\G_m\times \G_m$-torsor over $S$, which we recognize as having the structure of
$H$-torsor over $T$.

Let $N$ denote the homomorphism $H\to \G_m$ given by multiplication of the
nonzero matrix entries.
We define $G$ to be the kernel of $N$:
\[ 1\to G\to H\stackrel{N}\to \G_m\to 1. \]
Correspondingly we have
\[ \pi_*B\G_m\cong BH\to B\G_m, \]
sending $\Z/2\Z$-torsor $S\to T$ with line bundle $L$ on $S$ to the
\emph{norm} $N_{T/S}(L)$.
So, $BG$ classifies $\Z/2\Z$-torsors with line bundle and
trivialization of the norm line bundle.
\emph{To give a $G$-torsor over a scheme $T$ is the same
as to give a $\Z/2\Z$-torsor $S\to T$, a line bundle $L$ on $S$, and an
isomorphism $L\otimes \sigma^*L\cong \cO_S$, invariant under the
covering involution $\sigma\colon S\to S$.}

The $2$-dimensional representation of $H$, given by $H\subset GL_2$,
associates to the $H$-torsor determined by $S\to T$ (with
covering involution $\sigma$) and $L$, the rank $2$ vector bundle over $T$,
obtained by descent under $S\to T$ from $L\oplus \sigma^*L$ over $S$
with $\sigma^*(L\oplus \sigma^*L)\cong L\oplus \sigma^*L$.
The corresponding $2$-dimensional representation of $G$ admits the
same description and will be denoted by $\rho_2$.

There is the exact sequence
\begin{equation}
\label{eqn.mu2GG}
1\to \mu_2\to G\to G\to 1,
\end{equation}
where the homomorphism $G\to G$ is given by squaring the matrix entries.
The representation $\rho_2$ of the middle group $G$ in \eqref{eqn.mu2GG}
determines a \emph{projective} representation
\[ \omega\colon G\to PGL_2, \]
of the group $G$ appearing on the right in \eqref{eqn.mu2GG};
concretely,
\[
\begin{pmatrix} \alpha & 0 \\ 0 & \alpha^{-1} \end{pmatrix}\stackrel{\omega}\mapsto
\left[ \begin{pmatrix} \alpha & 0 \\ 0 & 1 \end{pmatrix} \right],
\qquad
\begin{pmatrix} 0 & \beta \\ \beta^{-1} & 0 \end{pmatrix}\stackrel{\omega}\mapsto
\left[ \begin{pmatrix} 0 & \beta \\ 1 & 0 \end{pmatrix} \right].
\]
As well, $\det(\rho_2)$ and $\rho_2^\vee\otimes \rho_2$ determine \emph{linear}
representations $\chi$, respectively $\rho_4$, of the group $G$ on the right.
There is the trace homomorphism $\rho_4\to 1$ to a trivial
one-dimensional representation.

\begin{prop}
\label{prop.trace}
Over $\Spec(\Z[1/2])$ we have an isomorphism of representations
\[ \rho_4\cong 1\oplus \chi\oplus \rho_2 \]
of $G$ such that projection to the first factor gives the trace homomorphism.
\end{prop}

\begin{proof}
Direct computation.
\end{proof}

The kernel $\chi\oplus \rho_2$ of the trace homomorphism in
Proposition \ref{prop.trace} will be denoted by $\rho_3$.

\begin{exam}
\label{exa.secondZ2torsor}
Let $T$ be a scheme over $\Spec(\Z[1/2])$ with $\Z/2\Z$-torsor $S\to T$ and
covering involution $\sigma\colon S\to S$,
and let $S'\to T$ be a second $\Z/2\Z$-torsor, to which there is an
associated line bundle $L_0$ with $L_0^{\otimes 2}\cong \cO_T$.
Let $L$ denote the pullback of $L_0$ to $S$, so we have
$\sigma^*L\cong L$, canonically.
The identification
\[ L\otimes \sigma^*L\cong L^{\otimes 2}\cong \cO_S \]
determines a $G$-torsor, which arises from the
pair of $\Z/2\Z$-torsors under the identification of the subgroup
of $G$, generated the $2\times 2$-permutation matrices and $\mu_2$,
with the Klein four-group.
\end{exam}

\subsection{Root stacks and conic bundles}
\label{sec.rootconic}
We recall and introduce notation for three flavors of root stack:
\begin{itemize}
\item[(i)] root stack $\sqrt{L}$
of a line bundle $L$ on a smooth algebraic variety $S$ over $k$,
a smooth Deligne-Mumford stack that is Zariski locally over $S$ isomorphic
to a product with the classifying stack $B\mu_2$
\cite[Defn.\ 2.2.6]{cadman} \cite[\S B.1]{AGV};
\item[(ii)] 
root stack $\sqrt{(S,D)}$ of a divisor $D\subset S$,
a Deligne-Mumford stack, locally for $D$ defined by
$f=0$ on affine open $\Spec(A)$, given by 
$$
[\Spec(A[z]/(z^2-f))/\mu_2],
$$ 
where
$\mu_2$ acts by scalar multiplication on $z$ \cite[Defn.\ 2.2.1]{cadman} \cite[\S B.2]{AGV};
\item[(iii)] 
iterated root stack 
$$
\sqrt{(S,\{D,D'\})}=\sqrt{(S,D)}\times_S\sqrt{(S,D')}
$$
of a pair of divisors $D$, $D'\subset S$
\cite[Defn.\ 2.2.4]{cadman} (or more generally of an arbitrary collection
$\{D_1,\dots,D_N\}$ of divisors on $S$).
\end{itemize}
The root stack in (ii) is smooth if $D$ is smooth, and in (iii) is smooth if
$D$ and $D'$ are smooth and intersect transversally
(or more generally if $D_1\cup \dots\cup D_N$ is a simple normal crossing divisor).
It is useful to be aware of a fourth flavor, defined by
Matsuki and Olsson \cite[Thm.\ 4.1]{matsukiolsson}, which for a
normal crossing divisor $D\subset S$ is a smooth Deligne-Mumford stack
with morphism to $S$ that \'etale locally over $S$ is of the form described in (iii).

Let $S$ be a smooth algebraic variety over $k$ and
$D\subset S$, $D=D_1\cup\dots\cup D_N$ a simple normal crossing divisor
with $D_i\cap D_{i'}\cap D_{i''}=\emptyset$ for distinct $i$, $i'$, $i''$.
One way that a standard conic bundle over $S$ may arise is
by modifying a smooth $\PP^1$-fibration
\begin{equation}
\label{eqn.P}
P\to \sqrt{(S,\{D_1,\dots,D_N\})}.
\end{equation}
\emph{We assume that the corresponding class in the Brauer group $\Br(S\setminus D)$
is nontrivial, obstructed for every $i$
from extending to a neighborhood of the generic point of $D_i$ by the
class of a nontrivial degree $2$ covering $\widetilde{D}_i\to D_i$ that is
furthermore assumed to be ramified over
the generic point of every component of $D_i\cap D_{i'}$ for all $i'\ne i$.}
Following \cite{oesinghaus} and appealing to
\cite[Rmk.\ 2.3]{kreschtschinkelmodels}, associated with \eqref{eqn.P}
there is a standard conic bundle $V\to S$.
The construction over $S\setminus D^{\mathrm{sing}}$ is that of
\cite[Prop.\ 3.1]{kreschtschinkelmodels}: blow up the locus in
$P$ with $\mu_2$-stabilizer, contract, and descend to
$S\setminus D^{\mathrm{sing}}$.
There is a unique extension to a standard conic bundle over $S$.

\begin{exam}
\label{exa.conicbundleA2}
When $S=\A^2$ and $D=D_1\cup D_2$, union of coordinate axes, we have
\[ \sqrt{(S,\{D_1,D_2\})}\cong [\A^2/\mu_2\times \mu_2] \]
where, writing $\kappa$, $\lambda$ for coordinates on $\A^2$, the
action of the first factor $\mu_2$ is by scalar multiplication on $\kappa$
and of the second factor $\mu_2$, by scalar multiplication on $\lambda$;
here, $S$ is identified with $\Spec(k[\kappa^2,\lambda^2])$.
Let
\[ P:=[\A^2\times \PP^1/\mu_2\times \mu_2], \]
where the action of the first factor $\mu_2$ on $\PP^1$ is by
permutation of the coordinates and of the second factor $\mu_2$, by
scalar multiplication on one of the coordinates.
The construction described above leads to $V$, defined
inside $S\times \PP^2$ with projective coordinates $x$, $y$, $z$ by
\[ \kappa^2 x^2+\lambda^2y^2-z^2=0. \]
The smooth $\PP^1$-fibration and conic bundle are identified over
$S\setminus D$ by
\[ (\kappa,\lambda,(\alpha:\beta))\mapsto
(\kappa^2,\lambda^2,(\lambda(-\alpha^2+\beta^2):2\kappa\alpha\beta:\kappa\lambda(\alpha^2+\beta^2))). \]
\end{exam}

\subsection{A useful conic bundle}
\label{sec.useful}
We work out in detail one instance of the construction of
Section \ref{sec.rootconic}.
Let $\Z/2\Z$ act on $\PP^1\times \PP^1$ by swapping the factors and consider
\[ W:=[\PP^1\times \PP^1/(\Z/2\Z)]. \]
Since $\PP^1\times \PP^1$ is, by
\[ ((u':v'),(u'':v''))\mapsto (\frac{1}{2}(u'v''+v'u''):u'u'':v'v''), \]
a double cover of $\PP^2$ branched over the conic in $\PP^2$ defined by
$t^2=uv$, we may view $W$ as the root stack of $\PP^2$ along the conic.
Let $H$ be the line $v=0$ in $\PP^2$ and
\[ X:=W\times_{\PP^2}\sqrt{\cO_{\PP^2}(-H)}. \]
The root stack $\sqrt{\cO_{\PP^2}(-H)}$ carries a tautological line bundle
whose tensor square is identified with the pulllback of $\cO_{\PP^2}(-H)$.
The tautological line bundle, pulled back to $X$, will be denoted by $M$;
its pullback to the degree $2$ \'etale cover $(\PP^1\times \PP^1)\times_{\PP^2}\sqrt{\cO_{\PP^2}(-H)}$ will be denoted by $M'$.
The pre-image of $H$ in $\PP^1\times \PP^1$ is a union $H'\cup H''$,
where $H'$ is defined by $v'=0$ and $H''$, by $v''=0$.
Now we let
\[ L:=M'\otimes_{\cO_{\PP^1\times \PP^1}} \cO_{\PP^1\times \PP^1}(H'). \]
So $\sigma^*L\cong M'\otimes_{\cO_{\PP^1\times \PP^1}} \cO_{\PP^1\times \PP^1}(H'')$,
where $\sigma$ denotes the covering involution, and
\[ L\otimes \sigma^*L\cong M'^{\otimes 2}\otimes_{\cO_{\PP^1\times \PP^1}}
\cO_{\PP^1\times \PP^1}(H'+H'')\cong
\cO_{(\PP^1\times \PP^1)\times_{\PP^2}\sqrt{\cO_{\PP^2}(-H)}}. \]
These data determine a $G$-torsor over $X$,
hence via the representation $\omega$ a smooth $\PP^1$-fibration
\[ F\to X. \]
By base change to the line bundle $M$ we obtain
\begin{equation}
\label{eqn.PtoM}
P:=F\times_XM\to M.
\end{equation}

\begin{prop}
\label{prop.P}
We write $t_0$, $u_0$, $v_0$ for affine coordinates on $\A^3$ and view
the blow-up $B\ell_0\A^3$ as subvariety of $\A^3\times \PP^2$, where $\PP^2$
has projective coordinates $t$, $u$, $v$.
\begin{itemize}
\item[(i)] Let $D\subset \A^3$ be the divisor $u_0v_0-t_0^2=0$.
Then 
\[M\cong \sqrt{(B\ell_0\A^3,\{D',E\})}, \]
where $D'$ denotes the
proper transform of $D$ and $E$, the exceptional divisor.
\item[(ii)] Let us write $\theta$ for the canonical section of
$\cO_{B\ell_0\A^3}(-1)$, vanishing on $E$.
The conic bundle construction of Section \ref{sec.rootconic} applied to
\eqref{eqn.PtoM} yields the conic bundle
\[ V\subset \PP\big(\cO_{B\ell_0\A^3}(1)\oplus
\cO_{B\ell_0\A^3}\oplus \cO_{B\ell_0\A^3}\big) \]
defined by the symmetric morphism of vector bundles
\begin{align*}
\big(\cO_{B\ell_0\A^3}(1)\oplus\cO_{B\ell_0\A^3}\oplus \cO_{B\ell_0\A^3}\big)&
\otimes \cO_{B\ell_0\A^3}(-1)
\stackrel{\begin{pmatrix} \theta & 0 & 0 \\ 0 & -u & -t \\ 0 & -t & -v \end{pmatrix}}
{\relbar\joinrel\relbar\joinrel\relbar\joinrel\relbar\joinrel\relbar\joinrel\relbar\joinrel\longrightarrow} \\
&\cO_{B\ell_0\A^3}(-1)\oplus
\cO_{B\ell_0\A^3}\oplus \cO_{B\ell_0\A^3}.
\end{align*}
\end{itemize}
\end{prop}

\begin{proof}
Assertion (i) follows directly from the identification of $B\ell_0\A^3$ with the
total space of $\cO_{\PP^2}(-H)$ and two general facts about
root stacks: the root stack constructions (i)--(iii) of Section \ref{sec.rootconic}
commute with base change,
and the total space of the tautological line bundle on the root stack of a line bundle
is isomorphic to the root stack of the zero section as divisor in the
original line bundle.

It suffices to identify the outcome of the conic bundle construction of
Section \ref{sec.rootconic} with the conic bundle defined by the
matrix in (ii) over the complement of a curve in $B\ell_0\A^3$.
We do this over the union of two affine charts, one defined by the
nonvanishing of $v$ and the other, by the nonvanishing of $u$.

\emph{Chart 1}: coordinates $v_0$, $t_1$, $u_1$ with
$t_0=v_0t_1$ and $u_0=v_0u_1$.
The corresponding open substack of $M$ is isomorphic to
\[ [\A^2/(\Z/2\Z)]\times [\A^1/(\Z/2\Z)], \]
action by swapping the coordinates and multiplication by $-1$ of the coordinate
on respective factors.
Denoting coordinates by $u'$, $u''$, respectively $\lambda_1$, we have
\[ v_0=\lambda_1^2,\qquad t_1=\frac{1}{2}(u'+u''),\qquad u_1=u'u''. \]
We are in the situation of Example \ref{exa.secondZ2torsor},
where the first $\Z/2\Z$-torsor is $\A^2\times [\A^1/(\Z/2\Z)]$
and the second $\Z/2\Z$-torsor is $[\A^2/(\Z/2\Z)]\times \A^1$.
So the corresponding open substack of $P$ is isomorphic to
\begin{equation}
\label{eqn.ofP}
[\A^2\times \A^1\times \PP^1/(\Z/2\Z\times \Z/2\Z)],
\end{equation}
with action of first $\Z/2\Z$ factor by swapping the coordinates of $\A^2$ and
swapping the coordinates of $\PP^1$, and
action of second $\Z/2\Z$ factor by multiplication by $-1$ of the coordinate of $\A^1$
and multiplication by $-1$ of one of the coordinates of $\PP^1$.

\emph{Chart 2}: coordinates $u_0$, $t_2$, $v_2$ with
$t_0=u_0t_2$ and $v_0=u_0v_2$.
The corresponding open substacks of $M$ and of $P$ admit the
same description as in Chart 1.
With coordinates $v'$, $v''$ and $\lambda_2$ on respective factors of
$M\cong [\A^2/(\Z/2\Z)]\times [\A^1/(\Z/2\Z)]$, we have
\[ u_0=\lambda_2^2,\qquad t_2=\frac{1}{2}(v'+v''),\qquad v_2=v'v''. \]

\emph{Transition between charts}:
Given a $k$-scheme $T$, the data of a $T$-valued point of Chart 1
consist of a $\Z/2\Z$-torsor $S\to T$ with equivariant map
$S\to \A^2$ and a second $\Z/2\Z$-torsor $S'\to T$ with
equivariant map $S'\to \A^1$.
To land in the overlap with Chart 2 the morphism $S\to \A^2$ should
have image contained in $(\A^1\setminus \{0\})^2$.
The product of coordinates is invariant and thus determines an invertible
function $f$ on $T$.
We denote by $\sqrt{f}\to T$ the associated degree $2$ \'etale cover
and combine this with $S'$ to obtain $\Z/2\Z$-torsor
\[ \widetilde{S}':=S'\times^{\Z/2\Z} \sqrt{f}. \]
Then, with notation $\mathrm{mul}$ and $\mathrm{inv}$ for multiplication
and multiplicative inverse, respectively,
the Chart 2 data consist of $\Z/2\Z$-torsors and equivariant maps
\[ S\to T\qquad\text{with}\qquad S\to (\A^1\setminus \{0\})^2
\stackrel{\mathrm{inv}\times\mathrm{inv}}\longrightarrow (\A^1\setminus \{0\})^2
\subset \A^2 \]
and
\[ \widetilde{S}'\to T\quad\text{with}\quad
\widetilde{S}'\to \A^1\quad\text{induced by}\quad
S'\times\sqrt{f}\to \A^1\times \G_m\stackrel{\text{mul}}\to \A^1. \]

\emph{Conic bundle}:
We introduce new coordinate $\kappa_1=(1/2)(-u'+u'')$ on Chart 1 and
$\kappa_2=(1/2)(-v'+v'')$ on Chart 2.
On Chart 1 we use $t_1$, $\kappa_1$ in place of $u'$, $u''$, and
on Chart 2 we use $t_2$, $\kappa_2$ in place of $v'$, $v''$.
We have
\[ u_1=t_1^2-\kappa_1^2\qquad\text{and}\qquad v_2=t_2^2-\kappa_2^2. \]
Performing the change of coordinates leads to description of open substack
of $M$ as
\[ \A^1\times [\A^2/(\Z/2\Z\times \Z/2\Z)], \]
with coordinates $t_i$, respectively $\kappa_i$, $\lambda_i$
on Chart $i$ for $i\in \{1,2\}$.
For the open substack of $P$ over each chart, there is simply an extra
factor $\PP^1$ as in \eqref{eqn.ofP}.
We have, over each chart, exactly the situation of
Example \ref{exa.conicbundleA2} (up to an extra $\A^1$-factor), and
the conic bundle is thereby defined by an explicit equation in
projective coordinates $x_i$, $y_i$, $z_i$.
With the change of coordinates
\[ \tilde{z}_i=z_i-t_ix_i \]
we obtain the equation
\[ -u_1x_1^2+v_0y_1^2-\tilde{z}_1^2-2t_1x_1\tilde{z}_1=0,
\quad\text{resp.}\quad
-v_2x_2^2+u_0y_2^2-\tilde{z}_2^2-2t_2x_2\tilde{z}_2=0, \]
with map
\[ (t_i,\kappa_i,\lambda_i,(\alpha_i:\beta_i))\mapsto
(\lambda_i(-\alpha_i^2+\beta_i^2):2\kappa_i\alpha_i\beta_i:
(\kappa_i+t_i)\lambda_i\alpha_i^2+(\kappa_i-t_i)\lambda_i\beta_i^2). \]
Writing $\gamma$ for a square root of $u_1=u'u''$, we have
\[
\lambda_2=\gamma\lambda_1\qquad\text{and}\qquad
(\alpha_2:\beta_2)=(u'^{-1}\gamma \alpha_1:\beta_1),
\]
and it follows that the projective coordinates on the two charts are related by
\begin{equation}
\label{eqn.projectivechange}
(x_2:y_2:\tilde{z}_2)=(\tilde{z}_1:u_1^{-1}y_1:x_1).
\end{equation}
The conic bundle defined by the matrix in (ii) yields, on the
two charts, precisely the equations obtained above with
relation \eqref{eqn.projectivechange} between projective coordinates.
\end{proof}

Given a smooth variety, the \emph{elementary transformation}
of a projectivized vector bundle along a section
over a divisor is the outcome of blowing up the section and
contracting to the projectivization of a vector bundle, whose dual is
in a natural way a subsheaf of the dual of the
original vector bundle \cite{maruyama}.

\begin{prop}
\label{prop.elementarytransformation}
We adopt the notation of Proposition \ref{prop.P}.
The elementary transformation of
$$
\PP(\cO_{B\ell_0\A^3}(1)\oplus \cO_{B\ell_0\A^3}\oplus \cO_{B\ell_0\A^3})
$$
along the section $\PP(\cO_E(1)\oplus 0\oplus 0)$ over $E$ is
$B\ell_0\A^3\times \PP^2$.
Writing $x$, $y$, $z$ for projective coordinates on the $\PP^2$ factor,
the conic bundle $V\to B\ell_0\A^3$ transforms to the conic bundle in $B\ell_0\A^3\times \PP^2$
defined by
\[ x^2-u_0y^2-2t_0yz-v_0z^2=0, \]
obtained by base-change from the conic bundle
with this defining equation over $\A^3$.
\end{prop}

\begin{proof}
If $\mathcal{B}$ is the locally free coherent sheaf
$\cO_{B\ell_0\A^3}(1)\oplus \cO_{B\ell_0\A^3}\oplus \cO_{B\ell_0\A^3}$,
the elementary transformation applied to $\PP(\mathcal{B})$ with the indicated section
over $E$ yields $\PP(\mathcal{B}')$, where the dual $\mathcal{B}'^\vee$ sits in
an exact sequence
\[ 0\to \mathcal{B}'^\vee\to \mathcal{B}^\vee\to \cO_E(-1)\to 0. \]
So, $\mathcal{B}'\cong \cO_{B\ell_0\A^3}^3$.
A direct computation establishes the remaining assertion.
\end{proof}

Returning to the $G$-torsor over $X$ we are also interested in the
vector bundles associated with the representations given in Section \ref{sec.representations}.
Notation $R$ with subscript will be used for the associated vector bundle,
where the subscript indicates the representation.
When the generic stabilizer $\Z/2\Z$ of $X$ acts trivially on fibers, the
vector bundle descends to $W$ and the notation $Q$ with subscript will
be used for the vector bundle on $W$.

\begin{lemm}
\label{lem.rho2onX}
Let $\Delta$ denote the diagonal in $\PP^1\times \PP^1$ and $r\in \Delta$,
and let us use the same notation for the corresponding closed substacks of $W$,
or of $X$.
\begin{itemize}
\item[(i)] The line bundle $R_\chi$ is the pullback of line bundle $Q_\chi$ on $W$, with
\[ Q_\chi\cong \cO_W(\Delta-H'-H'') \]
and characterization up to isomorphism as nontrivial, $2$-torsion in $\Pic(W)$.
\item[(ii)] The vector bundle $R_{\rho_2}$ fits in an exact sequence
\[ 0\to M\to R_{\rho_2}\to M^\vee\otimes R_\chi\to
M^\vee\otimes R_\chi\otimes \cO_{\{r\}}\to 0 \]
of coherent sheaves on $X$.
\end{itemize}
\end{lemm}

\begin{proof}
There is no loss of generality in supposing $r$ to be the point of
intersection of $H'$ and $H''$.

The line bundle $R_\chi$ is the pullback of the nontrivial line bundle on
$B\Z/2\Z$, corresponding to the $\Z/2\Z$-torsor
\[ 
(\PP^1\times \PP^1)\times_{\PP^2}\sqrt{\cO_{\PP^2}(-H)}\to X.
\]
So $Q_\chi$ exists, corresponds analogously to the
$\Z/2\Z$-torsor $\PP^1\times \PP^1\to W$, and is $2$-torsion in $\Pic(W)$.
As well, $Q_\chi$ is nontrivial,
since we have at $r$ a copy of $B\Z/2\Z$ in $W$.
By the general description of the Picard group of a root stack
\cite[\S 3.1]{cadman},
$\Pic(W)$ is generated by classes pulled back from $\PP^2$ and the
class of $\cO_W(\Delta)$, a line bundle whose square is isomorphic to
the pullback of $\cO_{\PP^2}(2)$.
So there is a unique nontrivial $2$-torsion class in $\Pic(W)$, the
one indicated in the statement of (i).

For (ii), we start with the description of $R_{\rho_2}$ as
isomorphic after pullback to $W$ to $L\oplus \sigma^*L$
and characterized by descent by the isomorphism
$\sigma^*(L\oplus \sigma^*L)\cong L\oplus \sigma^*L$, as indicated
in the description of $\rho_2$ in Section \ref{sec.representations}.
In the present case, $L\oplus \sigma^*L$ contains a diagonal copy of $M'$,
yielding upon descent an injective homomorphism
\[ M\to R_{\rho_2} \]
of coherent sheaves on $X$.
The cokernel $K$ is locally free of rank $1$ away from $r$.
The reflexive hull $K^{\vee\vee}$ is a line bundle, whose isomorphism type
is identified, by taking determinants, as $M^\vee\otimes R_\chi$.
Furthermore, $K^{\vee\vee}/K\cong K^{\vee\vee}|_{\{r\}}$.
Combining this information, we obtain the $4$-term exact sequence given in the statement.
\end{proof}

\begin{prop}
\label{prop.coho}
The following cohomology groups vanish:
\[ H^i(X,R_{\rho_3}\otimes \cO_X(-H'-H'')) \qquad\text{and}\qquad
H^i(X,R_{\rho_3}\otimes M) \qquad\text{for all $i$}, \]
and $H^i(X,R_{\rho_3}\otimes (M^\vee)^{\otimes j})$ for all $i$, $j>0$.
\end{prop}

\begin{proof}
The direct image functor on quasi-coherent sheaves is exact,
for the morphisms $X\to W$ and $W\to \PP^2$.
So, each of the above cohomology groups may be identified with a
cohomology group of the direct image sheaf on $\PP^2$.
Using Lemma \ref{lem.rho2onX},
the direct image of $R_{\rho_3}\otimes \cO_X(-H'-H'')$ on $W$ is
$\cO_W(\Delta-2H'-2H'')$, and on $\PP^2$ is $\cO_{\PP^2}(-2)$, which has
vanishing cohomology groups.
For $R_{\rho_3}\otimes M$ we have direct image on $W$ that sits in an
exact sequence between $\cO_W(-H'-H'')$ and
the kernel of restriction to the fiber over $r$ of $\cO_W(\Delta-H'-H'')$.
But the stabilizer action on fiber over $r$ is nontrivial, hence the
sheaf $\cO_W(\Delta-H'-H'')$ and its kernel of restriction to the fiber over $r$
have the same direct image on $\PP^2$.
Consequently, on $\PP^2$ we obtain the sheaf whose cohomology we need to compute
in the middle of a short exact sequence with
$\cO_{\PP^2}(-1)$ on the left and on the right.
The vanishing of cohomology groups follows.
The computation is similar
for $R_{\rho_3}\otimes (M^\vee)^{\otimes j}$ when $j$ is odd, except that
$\cO_{\PP^2}(i)$ with some $i>0$ occurs left and right, and we obtain
the vanishing of all higher cohomology groups.
The case that $j$ is even reduces as well to the vanishing of higher cohomology
of $\cO_{\PP^2}(i)$ for $i\ge 0$; we omit the details.
\end{proof}

\subsection{Deformation result}
\label{sec.defo}
Working over an
algebraically closed field $k$ of characteristic different from $2$,
we prove the existence of families of conic bundles,
where the special fiber is a standard conic bundle over a rational surface
with nodal discriminant curve, and in all other fibers the
discriminant curve is smooth.

\begin{theo}
\label{thm.defo}
Let $(B,b_0)$ be a pointed curve,
$\mathcal{S}\to B$ a smooth projective family of rational surfaces,
and $\mathcal{D}\subset \mathcal{S}$ an irreducible divisor,
smooth except for ordinary double points at
$q_1$, $\dots$, $q_r$ in the fiber $\mathcal{S}_{b_0}$, for some
positive integer $r$.
Let $\widetilde{\mathcal{D}}\to \mathcal{D}$ be a finite morphism
of degree $2$, \'etale over $\mathcal{D} \setminus \{q_1,\dots,q_r\}$,
with $\widetilde{\mathcal{D}}$ smooth.
We suppose, further, that $\mathcal{D}_{b_0}$ is connected and
smooth except for nodes at $q_1$, $\dots$, $q_r$, and
$\mathcal{D}_b$ is smooth for all $b\ne b_0$.
Then, after replacing $(B,b_0)$ by an \'etale neighborhood, there exists a
variety $\mathcal{V}$ fitting into a commutative diagram
\[
\xymatrix@C=8pt{
\mathcal{V}\ar[rr]^\varphi\ar[dr]_\psi && \mathcal{S} \ar[dl] \\
& B 
}
\]
such that $\varphi$ is a conic bundle (flat, projective, generically smooth,
all fibers are conics),
$\psi$ is smooth, and for every $b\in B$ the fiber $\mathcal{V}_b$ is
a standard conic bundle over $\mathcal{S}_b$ with corresponding
ramification data
$\widetilde{\mathcal{D}}_b\to \mathcal{D}_b$.
\end{theo}

\begin{rema}
\label{rem.MO}
The proof starts by replacing the fiber $\mathcal{S}_{b_0}$ by a root stack of the
sort defined by Matsuki and Olsson (cf.\ Section \ref{sec.rootconic}).
In the envisaged application, $\mathcal{D}_{b_0}$ consists of
two smooth curves meeting at the nodes $q_1$, $\dots$, $q_r$.
Then this root stack is just the iterated root stack of $\mathcal{S}_{b_0}$
along the two curves, and the subtle construction of Matsuki and Olsson,
based on logarithmic structures, is not needed.
\end{rema}

\begin{proof}
We follow the strategy of \cite{HKT}.
We construct, on the fiber over $c$, a suitable rank $2$ vector bundle over
a $\mu_2$-gerbe over a root stack of $\mathcal{S}_{b_0}$; this determines
$\mathcal{V}_{b_0}$.
We extend first the $\mu_2$-gerbe and then the vector bundle to the
whole family, allowing ourselves at each step to replace $(B,b_0)$ by
an \'etale neighborhood.

We denote the fiber $\mathcal{S}_{b_0}$ by $S$ and the fiber of the divisor
and cover $\widetilde{\mathcal{D}}_{b_0}\to \mathcal{D}_{b_0}$ by
$\widetilde{D}\to D$.
Since $D$ is a curve with nodes at $q_1$, $\dots$, $q_r$, the construction of
Matsuki and Olsson mentioned in Section \ref{sec.rootconic} yields a
smooth Deligne-Mumford stack $Y$ with flat morphism to $S$ that is an
isomorphism over $S\setminus D$.
Over smooth points, respectively over nodes of $D$ the stack $Y$ has stabilizer
$\mu_2$, respectively $\mu_2\times \mu_2$.
The cover $\widetilde{D}\to D$ determines a $2$-torsion element of
$\Br(S\setminus D)$, which is the restriction of a unique element of
$\Br(Y)$ (cf.\ \cite[Prop.\ 2 and Prop.\ 5]{HKT}).
There is, correspondingly, a $\mu_2$-gerbe
\[ Z\to Y, \]
which we may take over smooth points, respectively over nodes of $D$
to have stabilizer $\mu_2\times \mu_2$, respectively the dihedral group $\mathfrak D_4$,
with locally free sheaf $\mathcal{F}$ of rank $2$, such that the
projectivization of $\mathcal{F}$ is base-change to $Z$ of a smooth
$\PP^1$-fibration over $Y$, birational to a conic bundle
over $S$ with ramification data $\widetilde{D}\to D$.

By \cite[Prop.\ 17]{HKT},
after replacing $\mathcal{F}$ by a suitable locally free subsheaf
(obtained from $\mathcal{F}$ by elementary transformation over a
smooth divisor in $Y$, which may be taken in general position
with respect to the locus with nontrivial stabilizer),
we may suppose that the kernel of the trace homomorphism
\[
H^2(Y,(\mathcal{F}^\vee\otimes \mathcal{F})_0):=
\ker\big(H^2(Y,\mathcal{F}^\vee\otimes \mathcal{F})\to H^2(Y,\cO_Y)\big)
\]
vanishes; although $\mathcal{F}$ is a sheaf on the gerbe $Z$,
the locally free sheaf $\mathcal{F}^\vee\otimes \mathcal{F}$ descends to $Y$, and
this is meant by the above notation.
The space $H^2(Y,(\mathcal{F}^\vee\otimes \mathcal{F})_0)$ is the obstruction space for
the deformation theory of locally free coherent sheaves with
given determinant.

The Deligne-Mumford stack $Y$ does not sit in a smooth family with root stacks
$\sqrt{(\mathcal{S}_b,\mathcal{D}_b)}$.
After suitable modification, however, it sits in a flat family.

The singularities of $\mathcal{D}$ are resolved by blowing up:
\[ \mathcal{S}':=B\ell_{\{q_1,\dots,q_r\}}\mathcal{S}\qquad
\text{with smooth divisor}\qquad
\mathcal{D}':=B\ell_{\{q_1,\dots,q_r\}}\mathcal{D}. \]
Let $E=\bigcup_{i=1}^r E_i$ denote the exceptional divisor in
$\mathcal{S}'$.
The divisors $\mathcal{D}'$ and $E$ meet transversely, so the
iterated root stack
\[ \mathcal{Y}':=\sqrt{(\mathcal{S}',\{\mathcal{D}',E\})} \]
is smooth.
As a family over $B$, this has nonreduced fiber $\mathcal{Y}'_{b_0}$, but
\[ (\mathcal{Y}'_{b_0})_{\mathrm{red}}=Y' \cup \bigg(\bigcup_{i=1}^r X_i \bigg), \]
where $Y'$ denotes the blow-up of $Y$ at the points over $q_1$, $\dots$, $q_r$,
and each $X_i$ is a copy of the stack $X$
glued along the locus defined by the vanishing of the coordinate $t$,
in the notation of Section \ref{sec.useful}.
Let $Z'$ denote the corresponding blow-up of $Z$, so
$Z'\cong Y'\times_YZ$.
The coarse moduli space of $(\mathcal{Y}'_{b_0})_{\mathrm{red}}$ is
$\mathcal{S}'_{b_0}=S'\cup(\bigcup_{i=1}^r E_i)$, where $S'$ denotes the
blow-up of $S$ at $q_1$, $\dots$, $q_r$.

Over $X$ there is the $G$-torsor introduced in
Section \ref{sec.useful}, where $G$ is the algebraic group,
defined and exhibited as a $\mu_2$-extension of itself in
Section \ref{sec.representations}.
This way, we get a $\mu_2$-gerbe $U\to X$.
The substack in $X$ defined by $t=0$ is isomorphic to
$[\PP^1/(\Z/2\Z)]\times B\Z/2\Z$, and the
restriction of the $G$-torsor admits the description of
Example \ref{exa.secondZ2torsor}.
Consequently, the restriction of $U$ is isomorphic to
$[\PP^1/K]$,
where $K$ denotes the $\mu_2$-extension of the
copy of the Klein four-group in
$G$ generated by the $2\times 2$-permutation matrices and $\mu_2$,
acting through the Klein four-group.
Each exceptional divisor of $Z'$ admits the same description.
So, we may glue to obtain
\[ Z'\cup \bigg(\bigcup_{i=1}^r U_i \bigg), \]
a $\mu_2$-gerbe over $(\mathcal{Y}'_{b_0})_{\mathrm{red}}$, where each $U_i$ is a copy of $U$.
By applying the proper base change theorem for tame Deligne-Mumford stacks
\cite[App.\ A]{ACV} to the corresponding class
\[ \gamma\in H^2(\mathcal{Y}'_{b_0},\mu_2)=H^2((\mathcal{Y}'_{b_0})_{\mathrm{red}},\mu_2), \]
we obtain as in \cite{HKT} a class
\[ \Gamma\in H^2(\mathcal{Y}',\mu_2) \]
extending $\gamma$.
Accordingly there is a $\mu_2$-gerbe
\[ \mathcal{Z}'\to \mathcal{Y}' \]
with
\[ (\mathcal{Z}'_{b_0})_{\mathrm{red}}\cong Z'\cup\bigg(\bigcup_{i=1}^r U_i\bigg). \]

Corresponding to the representation $\rho_2$ of $G$ is a rank $2$ vector bundle
on $U$.
The restriction to the copy of $[\PP^1/K]$ over $t=0$ is determined by
the $2$-dimensional representation of $K\subset G$.
On the other hand, the restriction of $\mathcal{F}$ over a point $q_i$ is
associated with a linear representation of $\mathfrak D_4$, which after blowing up
becomes $[\PP^1/\mathfrak D_4]$ over the exceptional divisor,
in a manner that is compatible under $\mathfrak D_4\cong K$.
By identifying restrictions in this manner,
we obtain a rank $2$ vector bundle $\mathcal{F}'_0$ on $(\mathcal{Z}'_{b_0})_{\mathrm{red}}$ from
the pullback $\mathcal{F}'$ of $\mathcal{F}$ under $Z'\to Z$ and
the rank $2$ vector bundle on $U_i$ for every $i$.
Its determinant descends to a line bundle on
$(\mathcal{Y}'_{b_0})_{\mathrm{red}}$.

We argue that the determinant line bundle on $(\mathcal{Y}'_{b_0})_{\mathrm{red}}$
extends to a line bundle on $\mathcal{Y}'$,
after replacing $(B,b_0)$ by a suitable \'etale neighborhood.
It suffices to verify this after tensoring by the restriction of a line bundle
on $\mathcal{Y}'$.
Over $\mathcal{D}'$ there is a divisor in $\mathcal{Y}'$; we tensor by the
restriction of the associated line bundle.
This yields a line bundle on $(\mathcal{Y}'_{b_0})_{\mathrm{red}}$
with trivial stabilizer actions.
By \cite[Thm.\ 10.3]{alper}, this is isomorphic to the pullback of a
line bundle on the coarse moduli space $\mathcal{S}'_{b_0}$,
It thus suffices to show that every line bundle on
$\mathcal{S}'_{b_0}$ extends
(after replacing $(B,B_0)$ by a suitable \'etale neighborhood)
to a line bundle on $\mathcal{S}'$.
Tensoring by line bundles associated with the $E_i$ with appropriate multiplicities,
we are reduced to considering line bundles on $\mathcal{S}'_{b_0}$
pulled back from $S$, and thus to the corresponding fact for
line bundles on $\mathcal{S}$ and $S$, which is standard since
$\mathcal{S}$ is a smooth projective family of rational surfaces.

Now we show that the obstructions to extending $\mathcal{F}'_0$ to a vector bundle
on $\mathcal{Z}'$ vanish.
We define
\[ (\mathcal{Z}'_p)_{\mathrm{red}}=T_0\subset T'_0\subset T''_0\subset T_1\subset T'_1\subset \dots \]
to be the substacks of
$\mathcal{Z}'$ corresponding to the following effective divisors:
{\small
\[
\begin{array}{c|l|l}
&\text{divisor}&\text{obstruction space} \\ \hline
T_i&(2i+1)(X_1+\dots+X_r)+(i+1)Y'&H^2(X,R_{\rho_3}\otimes M)^r\\
T'_i&(2i+2)(X_1+\dots+X_r)+(i+1)Y'&H^2(Y',(\mathcal{F}'^\vee\otimes\mathcal{F}')_0)\\
T''_i&(2i+2)(X_1+\dots+X_r)+(i+2)Y'&H^2(X,R_{\rho_3}\otimes \cO_X(-H'-H''))^r
\end{array}
\]
}
In each case we have identified the obstruction space to extension of
locally free sheaf (on the gerbe) with fixed determinant.
The vanishing in each cases follows from
Proposition \ref{prop.coho} (cohomology on $X$) or the vanishing of
$H^2(Y,(\mathcal{F}^\vee\otimes \mathcal{F})_0)$.
So, starting with $\mathcal{F}'_0$ on the gerbe over $T_0$, we may extend
to successive infinitesimal neighborhoods and, by appealing
as in \cite{HKT} to the Grothendieck existence theorem
for tame Deligne-Mumford stacks \cite[Thm.\ A.1.1]{AV}, to a
locally free sheaf $\widetilde{\mathcal{F}}'$ on $\mathcal{Z}'$ extending $\mathcal{F}'_0$.

Associated with $\widetilde{\mathcal{F}}'$ is a
smooth $\PP^1$-fibration
\[ \mathcal{P}' \to \mathcal{Y}'. \]
The conic bundle construction of Section \ref{sec.rootconic} yields
a standard conic bundle
\[ \mathcal{V}'\to \mathcal{S}'. \]

We claim that with an elementary transformation and contraction, from
$\mathcal{V}'\to \mathcal{S}'$ we obtain $\mathcal{V}\to \mathcal{S}$,
fulfilling the requirements of the theorem.
We have already seen this in Proposition \ref{prop.elementarytransformation}
for the conic bundle $V\to B\ell_0\A^3$.
So, to complete the proof, it is enough to show for every $i$ that
$q_i\in \mathcal{S}$ and $0\in \A^3$ have a common \'etale neighborhood,
over which $\mathcal{V}'\to \mathcal{S}'$ and
$V\to B\ell_0\A^3$ are isomorphic.

Since $\mathcal{D}$ has an ordinary double point at $q_i$,
there is a diagram of \'etale morphisms of pointed schemes
\[
\xymatrix@C=12pt{
&(\mathcal{R},p)\ar[dl] \ar[dr] \\
(\mathcal{S},q_i) && (\A^3,0)
}
\]
such that $\mathcal{D}\subset \mathcal{S}$ and the divisor in
$\A^3$ defined by $u_0v_0-t_0^2=0$ have the same pre-image in $\mathcal{R}$.
Without loss of generality, $p$ is the only point of
$\mathcal{R}$ mapping to $q_i$ in $\mathcal{S}$, and is
the only point of $\mathcal{R}$ mapping to $0$ in $\A^3$.
As well, we may suppose that the image of the left-hand morphism
avoids $q_j$ for all $j\ne i$.
After blowing up, we obtain an analogous diagram with \'etale
morphisms
\[
\xymatrix@C=12pt{
&B\ell_p\mathcal{R}\ar[dl] \ar[dr] \\
\mathcal{S}' && B\ell_0\A^3
}
\]
There is a further diagram, in which each scheme is replaced by the
iterated root stack along the proper transform of the given divisor and the
exceptional divisor.

The class $\Gamma\in H^2(\mathcal{Y}',\mu_2)$ pulls back to a
class in $H^2$ of the iterated root stack of
$B\ell_p\mathcal{R}$.
As well, the $G$-torsor over $X$, pulled back to $M$ (Section \ref{sec.useful}),
pulls back to a class in the same $H^2$ group.
Since $\Gamma$ extends $\gamma$, whose construction also is based on the
$G$-torsor over $X$ of Section \ref{sec.useful},
we may conclude by applying the proper base change theorem to the
morphism to $\mathcal{R}$, that after possibly replacing
$(\mathcal{R},p)$ by an \'etale neighborhood, these two
$H^2$-classes coincide.
So, we obtain another analogous diagram of $\mu_2$-gerbes over
iterated root stacks.

Now we compare the vector bundles, obtained by pullback from
$\widetilde{\mathcal{F}}'$ on $\mathcal{Z}'$ and
the rank $2$ vector bundle on the $\mu_2$-gerbe over $M$, again coming from
the $G$-torsor over $M$ in Section \ref{sec.useful}.
Their restrictions to the copy of $U$ over $p$ are, again by construction,
isomorphic.
By an argument, analogous to above obstruction analysis but this time
using the vanishing of $H^1$, we see that an isomorphism over the copy of $U$
may be extended to all infinitesimal neighborhoods.
So, by the standard combination of the Grothendieck existence theorem
and Artin approximation, after replacing $(\mathcal{R},p)$ by an \'etale neighborhood
there is an isomorphism of the two vector bundles
over the gerbe over the iterated root stack of $B\ell_p\mathcal{R}$.
Consequently, the associated smooth $\PP^1$-fibrations are
isomorphic.
But the conic bundle construction of Section \ref{sec.rootconic}
commutes with \'etale base change of the underlying algebraic variety.
So the conic bundles
$\mathcal{V}'\times_{\mathcal{S}'}B\ell_p\mathcal{R}$
and $V\times_{B\ell_0\A^3}B\ell_p\mathcal{R}$ are isomorphic
over $B\ell_p\mathcal{R}$, as required.
\end{proof}

\section{Proof of Theorem~\ref{theo:headline}}
\label{sect:proof}
Let $\cS \ra B$ be as in Section \ref{subsect:surfaces}.
If $\cD\ra B$ is chosen in a general manner, then
$\cD$ will have smooth total space.
Note that after replacing $B$ by simply branched cover,
$\cD$ will satisfy the hypotheses in Theorem~\ref{thm.defo}.
So, we may apply Theorem~\ref{thm.defo} to obtain a family
$\cV \ra \cS \ra B$.
The composite $\cV \ra B$ is smooth, and
$\cV_{b_0}$ is stably rational; it remains to show that the very general fiber is not
stably rational.

We recall the main result of \cite{HKT}:
\begin{theo} 
\label{theo:conic}
Let $\mathcal L$ be a linear system of effective divisors on a smooth projective rational surface $S$,
with smooth and irreducible
general member.
Let $\cM$ be an irreducible component of the moduli space of pairs $(D, \tilde{D}\ra D)$, where $D\in \mathcal L$ is nodal and reduced
and $\tilde{D}\ra D$ is an \'etale cover of degree 2.
Assume that $\mathcal M$ contains a cover that is nontrivial over every irreducible component of a reducible curve
with smooth irreducible components.  
Then a conic bundle over $S$ corresponding to a very general point of $\mathcal M$ is not stably rational.
\end{theo}

We seek to apply this to standard conic bundles
$V\ra S$ over a quintic del Pezzo surface with degeneracy locus a 
generic bi-anticanonical divisor, as in Section~\ref{subsect:family}. 
The relevant reducible curve is a union of two generic anticanonical divisors, i.e.,
elliptic curves which admit non-trivial \'etale double covers.
The birational class of $V$ depends on the choice of a non-trivial \'etale double cover
$\tilde{D} \ra D$. We can invoke Theorem~\ref{theo:conic} once we know that the monodromy
action on such double covers is transitive. 

We apply the criterion of \cite[Thm.~3]{beaumonodromy}; see also \cite[Sect.~3]{HKT-moduli}.
We need to find members $D$, $D'\in |-2K_S|$ such that
\begin{itemize}
\item $D$ has an $E_6$ singularity;
\item $D'$ is a union of a pair of smooth curves meeting transversally in an odd number of points.
\end{itemize}
The second is obvious --- again use a pair of anticanonical curves. 
An $E_6$ singularity is locally analytically equivalent to $y^3=x^4$.
The plane sextic
\begin{align*}
\frac{1}{2}y^4z^2+y^3z^3
+\frac{1}{2}xy^4z+xy^3z^2
-{}&x^2y^3z-3x^2y^2z^2
-x^3y^3-2x^3y^2z+{}\\
&\frac{1}{2}x^3yz^2+3x^4y^2-\frac{1}{2}x^4yz+x^4z^2=0
\end{align*}
has $E_6$-singularity at $(0:0:1)$ and nodes at
four other general points, hence defines such a curve $D$.

\section{Limitations of this construction}
\label{sect:limit}

The stable rationality of smooth cubic threefolds is quite mysterious.
No stably rational examples are known but the known arguments
for disproving stable rationality fail in this case. Voisin
\cite{VoisinJEMS} has shown
that the existence of a decomposition of the diagonal for 
a cubic threefold $X$ reduces to
finding curves of `odd degree' in its intermediate Jacobian 
$\IIJ(X)$. Such curves
arise in many examples, which are therefore natural candidates for stable 
rationality.  

We would have liked to use the approach of Section~\ref{sect:geom} 
to exhibit a stably rational cubic threefold.
Following  the 
procedure of Section~\ref{subsect:embed}, for
a nonsingular plane cubic $C\subset \bP^2$ and point $p:=s_4\in \bP^2 \setminus C$,
we would need to have $p\in L$.
Yet $L$ must be the satellite conic, which 
contradicts Lemma~\ref{lemm:satellite}.

\bibliographystyle{plain}
\bibliography{stabratdeform}

\end{document}